\documentclass[12pt]{article}
\usepackage{}
\usepackage{color}
\usepackage{diagbox}
\usepackage{graphicx}
\usepackage{epsfig}
\usepackage{graphicx}
\usepackage[tbtags]{amsmath}
\usepackage{epsfig}
\usepackage{epstopdf}
\usepackage{amssymb}
\usepackage{bm}
\usepackage{amsmath}
\usepackage{amsthm}
\usepackage{cases}
\usepackage{amsfonts}
\usepackage{mathrsfs}
\usepackage{multirow}
\usepackage{tabularx}
\usepackage{color}
\usepackage{algorithm}
\newtheorem{thm}{Theorem}[section]
\newtheorem{lem}{Lemma}[section]
\newtheorem{rem}{Remark}[section]
\newtheorem{defn}{Definition}[section]

\newtheorem{alg}{Algorithm}[section]
\newtheorem{as}{Assumption}[section]

\numberwithin{equation}{section} \allowdisplaybreaks[4]

\begin{document}
\pagestyle{plain}
\title {The minimal hitting probability of continuous-time controlled Markov systems with countable states}
\author{Yanyun~Li \thanks{Yanyun~Li is with School of Mathematics, Sun Yat-Sen University, Guangzhou, 510275, China (email: liyy536@mail2.sysu.edu.cn).} and Junping Li
\thanks{Junping~Li(Corresponding author) is with Guangdong University of Science $\&$ Technology, Dongguan, 523083, China; Central South University, Changsha, 410083, China (email: jpli@mail.csu.edu.cn).}
}

\date{}
\maketitle \underline{}

{\bf Abstract:}
This paper concentrates on the minimal hitting probability of continuous-time controlled Markov systems (CTCMSs) with countable state and finite admissible action spaces. The existence of an optimal policy is first proved. In particular, for a special and important case of controlled branching processes (CBPs), it is proved that the minimal hitting probability is the unique solution to an improved optimal system of equations. Furthermore, a novel and precise improved-policy iteration algorithm of an optimal policy and the minimal hitting probability (minimal extinction probability) is presented for CBPs.
\vskip 0.2 in
\noindent{\bf Key Words.}
Continuous-time controlled Markov systems; Continuous-time controlled branching processes; Optimal policy; Improved-policy iteration algorithm; Minimal extinction probability.

\vskip 0.2 in \noindent {\bf Mathematics Subject Classification.}
91A15, 91A25

\setlength{\baselineskip}{0.25in}
\section{Introduction}

This paper concentrates on the minimal hitting probability of continuous-time controlled Markov systems (CTCMSs) with countable state spaces and finite admissible action spaces. Note that there are some references about the minimal hitting probability for discrete-time controlled Markov systems (DTCMSs), such as \cite{DM22,DEJ11,LGG22}. But it has not been considered in continuous-time with countable states. In this paper, we want  to minimize hitting probability to a prescribed set $B$, i.e., $h_i^{\pi}:=P_{i,0}^{\pi}(\tau_B<\infty)$ in  probability space $(\Omega,\mathcal{F},P^{\pi}_{i,u})$, where $P^{\pi}_{i,u}$ is the probability measure starting from state $i$ at time $u \geq 0$ under policy $\pi$ and $\tau_B$ is the first hitting time to $B$.

Our paper is inspirited by the discrete-time hitting probability problem in \cite{LGG22} as well as the literature on Markov decision processes (MDPs) with risk neutral criteria (\cite{A02,CR12,MLT05,S99}), risk probability criterion (\cite{BM04,BK95,HG-2020}) and risk-sensitive criterion (\cite{BR14,CH11,D07}), respectively. For risk neutral criteria of MDPs, the existence and algorithm of optimal policies have been studied, see \cite{H96,MLT05,S99} for discrete-time case and  \cite{GXP09,GHS12,GHZ15,GZ17,H96} for continuous-time case. In addition, for risk-sensitive criterion of MDPs, the existence and algorithms of optimal policies have been obtained in \cite{BR14,D07} for discrete-time case and in \cite{GL19} for continuous-time case. On the other hand, the risk probability criterion of MDPs has also been deeply discussed in \cite{BM04,BK95,HG-2020}.
In the above literature, the optimization criteria have a common feature that they are  defined by reward/cost functions. However, in reliability engineering, the hitting probability has no connection with reward/cost functions, see \cite{A07,DM22}. To minimize the hitting probability,  \cite{LGG22}  gives the existence and a typical algorithm of optimal policies under discrete-time with finite states. The existence and algorithm of optimal policies with minimal hitting probability in continuous-time case are also worth considering.

The problem of hitting probability originates from \cite{DEJ11}, whose purpose is maximizing the hitting probability to a given set while avoiding another set for DTCMSs with Borel state space and action space. In this paper, we will show the existence of an optimal policy for CTCMSs with countable states and finite actions, and present an algorithm of the optimal policy.
In proving the existence of an optimal policy, we find that the minimal hitting probability can be transferred to its embedded controlled Markov system. Note that the algorithm of the optimal policies in \cite{LGG22} is no longer available because the method of solving the non-uniqueness of solution to optimality system of equations (OE) in \cite{LGG22} is invalid in the case of countable state space.
However, for the important case of continuous-time CBPs, we show that the minimal hitting probability (called minimal extinction probability  in CBPs), is the unique solution of a new optimal system of equations, from which
a novel and improved iteration algorithm is presented.
Comparing with \cite{DM22}, although our algorithm involves solving an equation to find out a better policy set $F_m(a_*)$, no matter in which case, the policy evaluation in our policy iteration requires to solve only one equation (instead of the three equations in \cite{MLT05}), and the policy improvement here consists of only one phase (instead of the two phases in \cite{MLT05}).

The rest of this paper is organized as follows. In Section 2, we describe the optimization problem. The existence of the optimal policy is considered in Section 3. Concerning the algorithms, in Section 4 we present the algorithm of an optimal policy for continuous-time CBP. Section 5 contains the proofs of Theorems \ref{th4.1}-\ref{th4.2}.

\section{The optimization problem and basic conditions}\label{section}

The model considered in this paper is
\begin{eqnarray}\label{2.1}
	\{S, B, (A(i)\subset A: i \in S), q(j| i, a)\},
\end{eqnarray}
where the components are explained as below:
\begin{itemize}
	\item  ~$S$, the state space, is assumed to be countable;
	\item  $B$, a given finite target set, is a subset of $S$. In reliable engineering situations, $B$ represents the set of all failed states in a system;
	\item ~$A(i)$, the set of admissible actions at state $i \in S$ and assumed to be finite. $A:=\cup_{i\in S}A(i)$ is the action space. Then the set of all feasible state-actions pairs is
	\begin{eqnarray*}
		\mathbf{K}:=\{(i,a):\ i\in S,\ a\in A(i)\};
	\end{eqnarray*}
	
	\item $q(j|i, a)$ is the transition rate, which satisfies $q(j|i,a) \geq 0$ for all $(i,a) \in \mathbf{K}$ and $j \neq i$. Moreover, we assume that $q(j|i,a)$ are conservative, i.e.,
	\begin{eqnarray}\label{2.2}
		\sum_{j \in S}q(j|i,a)=0\quad \forall (i,a) \in \mathbf{K},
	\end{eqnarray}
	and stable, i.e.,
	\begin{eqnarray}\label{2.3}
		q^*(i):=\max_{a \in A(i)}q_i(a)< \infty\quad \forall i \in S,
	\end{eqnarray}
	where $q_i(a):=-q(i|i,a)< \infty$ for all $i \in \mathbf{K}$.
	
Note that if $q_i(a)=0$ for some $a \in A$, then $\sum\limits_{j \neq i}q(j|i,a)=0$, i.e., there is no transition at state $i$ under $a$. Therefore, we assume that $q_i(a)>0$ for all $i\in B^c$ and $a \in A(i)$.
\end{itemize}

Now we give the definition of a randomized Markov policy.
\begin{defn}\label{def2.1}
	{\rm A randomized Markov policy is a
		real-valued function $\pi_t(C|i)$ that satisfies the following conditions:
		\begin{description}
			\item[(i)]For all $i\in S$ and $C\in \mathscr{B}(A(i))$, the mapping $t\rightarrow \pi_t(C|i)$ is measurable on $[0,\infty)$.
			
			\item[(ii)] For all $i\in S$ and $t\geq 0$, $\pi_t(\cdot |i)$ is a probability measure on $\mathscr{B}(A)$ with $\pi_t(A(i)|i)=1$, where $\pi_t(C|i)$ is the probability that an action
			in $C$ is taken when the system's state is $i$ at time $t$.
		\end{description}
	}
\end{defn}
A randomized Markov policy $\pi_t(C|i)$ is said to be randomized stationary if $\pi_t(C|i)\equiv \pi(C|i)$. A randomized stationary policy $\pi$ is
said to be deterministic stationary if there exists a mapping $f$ from $S$ to $A$ such that $\pi(f(i)|i)\equiv1$ for all $i \in S$. Such deterministic stationary policy will be written as $f$ for brevity.

We denote $\Pi_m$ as the family of the randomized Markov policies and $F$ as the family of the deterministic stationary policies. Clearly, $\emptyset \neq F \subset \Pi_m$.

For each $\pi=\{\pi_t:\ t \geq 0\} \in \Pi_m$, the associated transition rates are defined as
\begin{eqnarray}\label{2.6}
	q(j|i, \pi_t):=\sum_{a \in A(i)}q(j|i,a)\pi_t(a|i)\ \ \text{for}\ i,j \in S, t \geq 0
\end{eqnarray}
and denote $q_i(\pi_t):=\sum_{a\in A(i)}q_i(a)\pi_t(a|i)$.
It follows from (\ref{2.2})-(\ref{2.3}) that (\ref{2.6}) is well-defined and bounded by $q^*(i)$. Moreover, $q(j|i, \pi_t)$ is measurable in $t \geq 0$ for any fixed $i,j \in S$, and $q(j|i, \pi_t)$ is also stable and conservative by (\ref{2.3}) and (\ref{2.6}).
In the deterministic stationary case, that is, $\pi_t=f \in F$, we write $q(j|i, \pi_t)$ as $q(j|i, f(i))$.

We now describe the evolution of the CTCMS given by (\ref{2.1}). Suppose that the system is at state $i \in S$ at time $t \geq 0$. The controller or decision-maker takes an action $a \in A(i)$. Then, a transition from state $i$ to state $j$ ($j \neq i$) occurs with probability $q(j|i,a)dt+o(dt)$; or the  system remains at state $i$ with probability $1+q_i(a)dt+o(dt)$. The  sojourn time at state $i$ is exponentially distributed with parameter $q_{i}(a)$. After transforming to state $j$, the similar evolution goes on like above.
The controller aims at finding the policy minimizing the probability of hitting the target set $B$.

In particular, if $S=\mathbf{Z}_+$ (the nonnegative integer set) and $q(j|i,a)=ib_{j-i+1}(a)\ (j\geq i-1), =0\ (j<i-1)$ for any $i\in \mathbf{Z}_+$ and $a\in A(i)$, where $b_k(a)\geq 0\ (k\neq 1), b_1(a)=\sum_{k\neq 1}b_k(a)<\infty$, then the CTCMS given by (\ref{2.1}) becomes a continuous-time controlled branching system.

By \cite{GXP09}, for each policy $\pi \in \Pi_m$, there exists a transition probability on $S$, denoted by $p_{\pi}(s,i,t,j)$ for $i,j \in S$ and $t \geq s \geq 0$, satisfying
\begin{eqnarray}\label{2.7}
	\frac{\partial p_{\pi}(s,i,t,j)}
	{\partial s}=
	-\sum_{k=0}^{\infty}q(k|i,\pi_s) p_{\pi}(s,k,t,j)
\end{eqnarray}
and
\begin{eqnarray}\label{2.8}
	\frac{\partial p_{\pi}(s,i,t,j)}{\partial t}=\sum_{k=1}^{\infty}p_{\pi}(s,i,t,k) q(j|k,\pi_t)
\end{eqnarray}
for all $i,j\in S$ and almost every $t\geq s\geq 0$, which are called Kolmogorov backward and forward equations, respectively.

To guarantee that the transition function $p_{\pi}(s,i,t,j)$ is regular (e.g., unique and honest),
we make the following simple and mild assumption which is due to \cite[P. 13]{GXP09}.
\begin{as}\label{as2.1}
	{\rm	There exists a function $R \geq 1$ on $S$ and constants $c_0 \neq 0$, $b_0 \geq 0$, and $L_0>0$ such that
		\begin{description}
			\item [(1)]\ $\sum_{j \in S}q(j|i,a)R(j) \leq c_0R(i)+b_0$ for all $(i,a) \in \mathbf{K}$;
			\item [(2)]\ $q^*(i) \leq L_0R(i)$ for all  $i \in S$.
	\end{description}}
\end{as}
Under Assumption \ref{as2.1}, the transition probability $p_{\pi}(s,i,t,j)$ is regular and thus unique.
Let $\{X(t):t\geq 0\}$ be the CTCMS described as above and $T_k$ be the $k$'th jumping time of it. Then, $\{X(T_k):k\geq 0\}$ is the corresponding embedded CMS. Also, $\{(X(t),a(t)):t\geq 0\}$ is the state-action process.

For any $i \in S$ and policy $\pi=\{\pi_t:\ t \geq 0\} \in \Pi_m$, Theorem 2.3 in \cite{GXP09} ensures the existence of a unique probability measure $P_{i,u}^{\pi}\ (u \geq 0)$ on $(\Omega, \mathscr{F})$ of the state-action process $\{(X(t),a(t)):t \geq 0\}$. Moreover, under a fixed policy $\pi=\{\pi_t:\ t \geq 0\} \in \Pi_m$,
\begin{eqnarray*}
	&&P_{i,0}^{\pi}(X(0)=i)=1,\ \quad
	P_{i,0}^{\pi}(X(t)=k|X(s)=j)=p_{\pi}(s,j,t,k),\\
	&&P_{i,0}^{\pi}(a(t)\in C|X(t)=j)=\pi_t(C|j)
\end{eqnarray*}
for all $t \geq s \geq 0$, $j \in S$ and $C \in \mathscr{B}(A)$.

Now, let
\begin{eqnarray}\label{tau}
	\tau_B=\inf \{t \geq 0: \ X(t) \in B\}
\end{eqnarray}
be the first hitting time on $B$.  In reliable engineering situations, $\tau_B$ represents the life-time of the system.
For a given policy $\pi \in \Pi_m$, the hitting probability from $i\in S$ at time $u \geq 0$, denoted by $h_{(u,i)}^{\pi},$ is given by
\begin{eqnarray}\label{hitting}
	h_{(u,i)}^{\pi}:=P_{i,u}^{\pi}(\tau_B< \infty)
\end{eqnarray}
and $h_{(0,i)}^{\pi}$ is simply rewritten as $h_i^{\pi}$.

The purpose of this paper is to consider the optimization problem:
\begin{eqnarray}\label{2.9}
	\mbox{Minimize} \ h_i^{\pi} \ \mbox{ over} \ \pi \in \Pi_m ,\quad {\rm for \ all}  \ i\in S.
\end{eqnarray}
From (\ref{tau})-(\ref{hitting}), $h^{\pi}_i=1\ (i\in B)$ for all $\pi \in \Pi_m$. Therefore, the above optimization problem (\ref{2.9}) is transformed to:
\begin{eqnarray}\label{2.10}
	\mbox{Minimize} \ h^{\pi}_i \ \mbox{ over} \ \pi\in \Pi_m ,\quad {\rm for \ all}  \ i\in B^c.
\end{eqnarray}

Since the main purpose of this paper is to consider the minimal hitting probability of target set $B$, we assume that $B$ is absorbing under all $\pi\in \Pi_m$.
Let
\begin{eqnarray}\label{2.11}
	h_i^*:=\inf_{\pi \in \Pi_m}h_i^{\pi},\ \  i\in B^c,
\end{eqnarray}
which refers as the minimal hitting probability starting from state $i \in B^c$.

\begin{defn}\label{def2.2}
	{\rm	A policy $\pi^*\in \Pi_m$ satisfying $
		h^{\pi^*}_i=h_i^*$ for all $i\in B^c$,
		is said to be optimal.}
\end{defn}

\section{On the existence of an optimal stationary policy}

In \cite{DEJ11,LGG22}, the existence of an optimal policy has been given for DTCMSs. However, in continous-time case, it still has not been shown, and
we will present the existence in this section. Before doing so, some preparation (Lemmas \ref{le3.1}-\ref{le3.2}) are presented, where Lemma \ref{le3.1}  follows directly from the Markov property and  Theorem 2.6 in \cite{Chen1992}.

\begin{lem}\label{le3.1}	
	{\rm For all $f \in F$,		
\begin{description}
	\item[(i)] $(h^f_i:i \in B^c)$ is a nonnegative solution of the following probabilistic equation system:
\begin{eqnarray}\label{3.9}
\begin{cases}
x_i=\frac{1}{q_i(f(i))}[\sum\limits_{k \in B}{q(k|i,f(i))}+\sum\limits_{k \in B^c \setminus \{i\}}q(k|i, f(i))x_k],\ \ & i\in B^c,\\
					0\leq x_i \leq 1,\ & i\in B^c;
\end{cases}
\end{eqnarray}	
			
\item[(ii)] if a vector $(y_i:\ i\in B^c)$ satisfies
\begin{eqnarray*}
\begin{cases}
y_i\geq \frac{1}{q_i(f(i))}[\sum\limits_{k \in B}{q(k|i,f(i))}+\sum\limits_{k \in B^c \setminus \{i\}}q(k|i, f(i))y_k],\ &  i\in B^c,\\
					y_i\geq 0,\ & i\in B^c,
\end{cases}
\end{eqnarray*}
then $y_i \geq h^f_i$ for all $i \in B^c$.
	
\end{description} }
\end{lem}

Next we derive some properties of $h^{\pi}_{(u,i)}$. Denote
$
p_{\pi}(s,i,t,B):=\sum\limits_{j \in B}p_{\pi}(s,i,t,j).
$
\begin{lem}\label{le3.2}
	{\rm Suppose that $\pi=\{\pi_t:t\geq 0\}\in \Pi_{m}$. For every $i \in B^c$,
\begin{description}
\item[(i)] $h^{\pi}_{(u,i)}$ is differentiable respect to $u\in [0,\infty)$;
			
\item[(ii)] $
			\inf\limits_{\pi\in \Pi_m}[q_i(\pi_0) (h^{\pi}_{(0,i)}-h^*_i)-(h^{\pi}_{(0,i)})']\leq 0$.
		\end{description}
	}
\end{lem}

\begin{proof}
 Since $B$ is absorbing, we have $P^{\pi}_{i,u}(\tau_B\leq t)=p_{\pi}(u,i,t,B)$ for all $i\in S$ and $t\geq u \geq 0$. By Proposition B.2 (c) in \cite{GXP09}, $\lim\limits_{t \rightarrow \infty}p_{\pi}(u,i,t,B)$ exists for all $u \geq 0$ and $i \in S$. Thus,
\begin{eqnarray}\label{lim}
h^{\pi}_{(u,i)}\!\!=\! \! P_{i,u}^{\pi}(\tau_B < \infty)\!\!=\lim_{t \rightarrow \infty}\! p_{\pi}(u,i,t,B)\qquad \forall u \geq 0.
\end{eqnarray}

Now prove (i). We first prove that $h^{\pi}_{(u,i)}$ is continuous in $u\in [0,\infty)$. Indeed, it follows from the absorption of $B$ and Proposition B.2 (c) in \cite{GXP09} that for any $\varepsilon\in (0,t-u)$,
 \begin{eqnarray*}
|p_{\pi}(u,i,t,B)-p_{\pi}(u+\varepsilon,i,t,B)|\leq|p_{\pi}
(u,i,u+\varepsilon,i)-1|+p_{\pi}(u,i,u+\varepsilon,\{i\}^c).
 \end{eqnarray*}
 Letting $t\rightarrow \infty$ and then letting $\varepsilon\downarrow 0$ yield the right-continuity of $h^{\pi}_{(u,i)}$. Similar argument implies the left-continuity of $h^{\pi}_{(u,i)}$.

 By (\ref{2.7}), for all $i\in B^c$ and almost every $t\geq u \geq 0$,
\begin{eqnarray}\label{ppi} p_{\pi}(u,\!i,\!t,\!B)\!-\!p_{\pi}(0,\!i,\!t,\!B)\!=\!-\sum_{k \in S}\!\!\int_0^u\!\!
q(k|i,\pi_v)
p_{\pi}(v,\!k,\!t,\!B)dv.
\end{eqnarray}	
Letting $t \rightarrow \infty$ on (\ref{ppi}) and using dominated convergence theorem implies that
\begin{eqnarray*}
h^{\pi}_{(u,i)}-h^{\pi}_i=-\sum_{k \in S}\int_0^uq(k|i,\pi_v)
h^{\pi}_{(v,k)}dv,\ \ i\in B^c,\ a.e.\ u\geq 0.
\end{eqnarray*}
By the continuity of both sides in the above equality, we know that it actually holds for all $u\in [0,\infty)$. Hence,
\begin{eqnarray}\label{3.2}
		(h^{\pi}_{(u,i)})'=-\sum_{k \in S}q(k|i,\pi_u)
		h^{\pi}_{(u,k)},\ \ i\in B^c,\ a.e.\ u\geq 0,
\end{eqnarray}
which implies (i).
	
Next prove (ii). Suppose that there exists some $i^* \in B^c$ such that
\begin{eqnarray*}
		\inf_{\pi\in \Pi_{m}}[q_{i^*}(\pi_0)(h^{\pi}_{(0,i^*)}-h^*_{i^*})-(h^{\pi}_{(0,i^*)})']>0.
\end{eqnarray*}
Denote $\tilde{q}(i^*):=\max\limits_{a\in A(i^*)}q_{i^*}(a)$. By (\ref{2.3}) and the assumption $q_i(a)>0\ (i\in B^c, a\in A(i))$, we know that $0<\tilde{q}(i^*)<\infty$. Moreover, there exists $\delta>0$ such that
	\begin{eqnarray}\label{as}
		\inf_{\pi\in \Pi_{m}}[q_{i^*}(\pi_0)(h^{\pi}_{(0,i^*)}-h^*_{i^*})-(h^{\pi}_{(0,i^*)})']>\delta>0,
	\end{eqnarray}
and thus for any $\pi \in \Pi_m$, $\tilde{q}(i^*)(h^{\pi}_{(0,i^*)}-h^*_{i^*})-(h^{\pi}_{(0,i^*)})'>\delta>0$. It follows from (\ref{2.11}) that for any $\varepsilon \in (0,\delta)$, there exists $\tilde{\pi} \in \Pi_m$ such that
\begin{eqnarray}\label{hh}
h^{\tilde{\pi}}_{(0,i^*)}-h^*_{i^*}<\tilde{q}(i^*)^{-1}\varepsilon.
\end{eqnarray}
Denote
$\Pi_{\varepsilon}:=\{\pi\in \Pi_m\!: h^{\pi}_{(0,i^*)}-h^*_{i^*}<\tilde{q}(i^*)^{-1}\varepsilon\}$.
Obviously, by (\ref{hh}), $\Pi_{\varepsilon} \neq \emptyset$. It follows from (\ref{as}) that
\begin{eqnarray}\label{3.3}
(h^{\pi}_{(0,i^*)})'<-\delta+\varepsilon<0,\quad \forall \pi\in \Pi_{\varepsilon}.
\end{eqnarray}

Take a policy $\tilde{\pi}=\{\tilde{\pi}_t:t\geq 0\}\in \Pi_{\varepsilon}$ and let
\begin{equation}\label{3.4}
		m_{\tilde{\pi}}:=\inf\{u\geq 0:\ (h^{\tilde{\pi}}_{(u,i^*)})'\geq -\delta+\varepsilon\}.
\end{equation}
Obviously, $m_{\tilde{\pi}}<\infty$. Therefore, $(h^{\tilde{\pi}}_{(u,i^*)})'<-\delta+\varepsilon<0$ for all $u \in [0,	m_{\tilde{\pi}})$.
	
Now, take $\pi^{(v)}=\{\pi^{(v)}_{t}: t \geq 0\}=\{\tilde{\pi}_{m_{\tilde{\pi}}+v+t}:t\geq 0\}\in \Pi_m$ for $v\geq 0$.
By the homogeneous property, we have
	\begin{eqnarray*}	p_{\pi^{(v)}}(u,i,t,j)=p_{\tilde{\pi}}(m_{\tilde{\pi}}+v+u,i,
m_{\tilde{\pi}}+v+t,j),\ \ i,j\in S,\ 0\leq u\leq t
	\end{eqnarray*}
	and
\begin{equation}\label{h}	h^{\pi^{(v)}}_{(u,i^*)}=h^{\tilde{\pi}}_{(m_{\tilde{\pi}}+v+u,i^*)},\quad u \geq 0.
\end{equation}
Thus, $\pi^{(0)}\in \Pi_{\varepsilon}$. It follows from (\ref{h}) that
\begin{eqnarray*}
h^{\pi^{(v)}}_{(0,i^*)}=h^{\tilde{\pi}}_{(m_{\tilde{\pi}}+v,i^*)},\ v\geq 0.
\end{eqnarray*}
Note that $h^{\tilde{\pi}}_{(m_{\tilde{\pi}},i^*)}-h^*_{i^*}\leq h^{\tilde{\pi}}_{(0,i^*)}-h^*_{i^*}<\tilde{q}(i^*)^{-1}\varepsilon$. By the continuity of $h^{\tilde{\pi}}_{(u,i^*)}$, there exists $\tilde{v}>0$ such that $h^{\pi^{(v)}}_{(0,i^*)}-h^*_{i^*}=h^{\tilde{\pi}}_{(m_{\tilde{\pi}}+v,i^*)}-h^*_{i^*}
<\tilde{q}(i^*)^{-1}\varepsilon$ for all $v\in[0,\tilde{v})$. Therefore, $\pi^{(v)}\in \Pi_{\varepsilon}$ for all $v\in[0,\tilde{v})$ and hence by (\ref{3.3}),
\begin{eqnarray*}
(h^{\tilde{\pi}}_{(m_{\tilde{\pi}}+v,i^*)})'=(h^{\pi^{(v)}}_{(0,i^*)})'
<-\delta+\varepsilon,\quad \forall v\in [0,\tilde{v})
\end{eqnarray*}
which contradicts with the definition of $m_{\tilde{\pi}}$. (ii) is proved.
\end{proof}
We now show the existence of an optimal policy for CTCMSs.

\begin{thm}\label{th3.1}
	{\rm \begin{description} \item[(i)] The minimal hitting probability $(h^*_i:\ i \in B^c)$ satisfies the following optimality equation (OE):
\begin{eqnarray*}
				x_i= \min_{a \in A(i)}\{\frac{1}{q_i(a)}[\sum_{k \in B}q(k|i, a)+\sum_{k \in B^c \setminus \{i\}}q(k|i,a)x_k]\},\ i\in B^c.
\end{eqnarray*}
			
\item[(ii)] There exists a policy $f^* \in F$ such that
			\begin{eqnarray*}
				f^*(i) \in \mathop{\arg\min}_{a\in A(i)}\left\{\frac{1}{q_i(a)}[\sum_{k \in B}q(k|i, a)+\sum_{k \in B^c \setminus \{i\}}q(k|i,a)h_k^*]\right \},\ \ i \in B^c,
			\end{eqnarray*}
			and $f^*$ is optimal.
\end{description}	}
\end{thm}
\begin{proof}
	{\rm	For every $\pi=\{\pi_t:t\geq 0\}\in \Pi_m$, it follows from (\ref{3.2}) and (\ref{2.11}) that
		\begin{eqnarray*}
			&&-(h^{\pi}_{(0,i)})'+q_i(\pi_0)
			h^{\pi}_i\\
			&=& \sum_{k \in B}q(k|i,\pi_0)+\sum_{k \in B^c \setminus \{i\}}q(k|i,\pi_0)
			h^{\pi}_k\\
			&\geq & \sum_{k \in B}q(k|i,\pi_0)+\sum_{k \in B^c \setminus \{i\}}q(k|i,\pi_0)
			h^*_k,\ \ i\in B^c,
		\end{eqnarray*}
i.e.,
\begin{eqnarray*}
q_i(\pi_0)(h^{\pi}_i-h^*_i)-(h^{\pi}_{(0,i)})'\geq \sum_{k \in B}q(k|i,\pi_0)+\sum_{k \in B^c}q(k|i,\pi_0)
			h^*_k,\ \  i \in B^c.
\end{eqnarray*}
Taking infimum over $\pi \in \Pi_m$ on the both sides and using Lemma~\ref{le3.2}(ii) yield that
\begin{eqnarray*}
\inf_{\pi\in \Pi_m}[\sum_{k \in B}q(k|i,\pi_0)+\sum_{k \in B^c}q(k|i,\pi_0)h^*_k]\leq \inf_{\pi \in \Pi_m}[q_i(\pi_0)(h^{\pi}_i-h^*_i)-(h^{\pi}_{(0,i)})'] \leq 0,\ \ i\in B^c.
\end{eqnarray*}
Since $A(i)$ is finite, there exists $a^*_i \in A(i)$ for every $i \in B^c$ such that
\begin{eqnarray*}
\sum_{k \in B}q(k|i,a^*_i)+\sum_{k \in B^c}q(k|i,a^*_i)
			h^*_k=\min_{a\in A(i)}[\sum_{k \in B}q(k|i,a)+\sum_{k \in B^c}q(k|i,a)
	h^*_k]\leq 0,\ \ i\in B^c.
\end{eqnarray*}
Take a policy $f^*$ satisfying
\begin{eqnarray*}
\begin{cases}
f^*(i)=a_i^*,\ & \ \text{if} \ i \in B^c\\
f^*(i) \in A(i),\ & \ \text{if} \ i \in B.
\end{cases}
\end{eqnarray*}
Then $f^* \in F$ (i.e., $f^*$ is a stationary Markov policy) and
\begin{eqnarray*}
\sum_{k \in B}q(k|i,f^*(i))+\sum_{k \in B^c}q(k|i,f^*(i))
h^*_k\leq 0,\ \ i\in B^c,
\end{eqnarray*}
i.e.,
\begin{eqnarray}\label{h^*}
h_i^* \geq \frac{1}{q_i(f^*(i))}[\sum_{k \in B}q(k|i,f^*(i))+\!\sum_{k \in B^c \setminus \{i\}}q(k|i,f^*(i))h^*_k],\ \ i \in B^c.
\end{eqnarray}
By Lemma \ref{le3.1}(i), we yield
\begin{eqnarray*}
h_i^{f^*}=\frac{1}{q_i(f^*(i))}[\sum_{k \in B}q(k|i,f^*(i))+\sum_{k \in B^c \setminus \{i\}}q(k|i,f^*(i))h_k^{f^*}],\ \ i\in B^c,
\end{eqnarray*}
from which, together with Lemma \ref{le3.1}(ii) and (\ref{h^*}), we obtain
$h^{f^*}_i\leq h^*_i$ for all $i \in B^c$.
Therefore, $h^{f^*}_i = h^*_i$ for all $i \in B^c$ (since $h_i^*=\inf_{\pi \in \Pi_m}h_i^{\pi}$). Hence,
\begin{eqnarray*}
h^*_i&\!\!=&\!\! h^{f^*}_i \!\!=\!\! \frac{1}{q_i(f^*(i))}[\sum_{k \in B}q(k|i,f^*(i))\!+\!\!\sum_{k \in B^c \setminus \{i\}} \! q(k|i,f^*(i))h^{f^*}_k]\\
			\!\!&=&\!\!\frac{1}{q_i(f^*(i))}[\sum_{k \in B}q(k|i,f^*(i))+\sum_{k \in B^c \setminus \{i\}}q(k|i,f^*(i))h^*_k]\\
			&=& \min_{a \in A(i)}\{\frac{1}{q_i(a)}[\sum_{k \in B}q(k|i,a)+\sum_{k \in B^c \setminus \{i\}}q(k|i,a)h^*_k]\},\  i\in B^c.
\end{eqnarray*}
(i) is proved. (ii) follows directly from (i).
	}
\end{proof}

\begin{rem}\label{re3.2}
\rm{(i)\ Theorem~\ref{th3.1} reveals that the minimal hitting probability $(h_i^*:\ i \in B^c)$ satisfies the OE and it can be found in $F$. Therefore, we only need to pay our attention to policies in $F$. But there also has the same problem in \cite{LGG22} that $(h_i^*:\ i \in B^c)$ is the minimal nonnegative solution of (\ref{3.9}) other than the unique one. Generally, for infinitely countable state space, (\ref{3.9}) may have multiple solutions since it involves infinite equations. However, in a special CBP considered in Section 4, such problem can be avoided.

(ii)\ Since the transition probability of the related embedded CMS is given by
\begin{eqnarray*}
	\tilde{p}(k|i,f(i)):=\begin{cases}\frac{q(k|i, f(i))}{q_i(f(i))},\ & k\neq i\\
		0,\ & k=i\\
	\end{cases}
\end{eqnarray*}
for fixed $f \in F$, by Theorem~\ref{th3.1}, the optimization problem of CTCMSs can be transformed to the optimization problem of its embedded CMSs.}
\end{rem}

\section{Application to branching processes}\label{section4}
Although the hitting behavior of CTCMSs can be transformed to discrete-time analogous, the algorithm of an optimal policy in \cite{LGG22} can not terminate in a finite number of iterations for countable state space. In this section, we will show that the minimal hitting probability for the special and important model of CBPs can be obtained precisely.

Branching process is an important class of stochastic models, see \cite{AN72,C21,EV-14,ES-17,H63,WC-18}. Recall that for a branching process, by \cite[P. 1-2]{AN72}, at any time, each particle splits independently of others, and thus from \cite[P. 96-97]{H63}, denote the transition rate as
\begin{eqnarray}\label{q(j|i)}
	q(j|i)=\begin{cases}
		ib_{j-i+1},\ & \text{if}\ j\geq i-1\\
		0,\ & \text{otherwise},
	\end{cases}
\end{eqnarray}
where $\{b_k:\ k \geq 0\}$ is the branching mechanism
satisfying $
b_1 \leq 0,\ b_k \geq 0\ (k \neq 1)$ and $ \sum\limits_{k=0}^{\infty}b_k=0$.
The generating function is defined as $
B(v):=\sum\limits_{k=0}^{\infty}b_{k}v^k$ for all $v \in [-1,1]$.

\subsection{The uniqueness of OE in branching case}
In continuous-time CBPs, the branching mechanisms are determined by action $a \in A$, and thus we denote it as $\{b_k(a):\ k \geq 0\}$. Hence, the transition rate and generating function are denoted by $q(j|i,a)$ and $B(a;v)$ for every $a \in A$, respectively. Obviously, for an action $a \in A$, if $\sum\limits_{k=2}^{\infty}b_k(a)=0$, then there is no offspring under $a \in A$. In order to avoid such trivial case, we can assume that $\sum\limits_{k=2}^{\infty}b_k(a)>0$ for all $a\in A$.

Suppose that there exists an integer $m \geq 1$ such that
\begin{eqnarray}\label{A}
	A(i)=\bar{A},\ \ i \geq m.
\end{eqnarray}
Because CBPs often describe the evolution of population, it is worth to consider the case that $B:=\{0\}$, the hitting probability $h_i^f$ is just the extinction probability $ep_i^f:=P^f(\tau_0<\infty\mid X(0)=i)$ in CBPs.
Therefore, (\ref{2.10}) can be transferred into calculating the minimal extinction probability,
\begin{eqnarray*}
	ep^*_i:=\min_{f \in F}P^f(\tau_0<\infty\mid X(0)=i),\quad i \geq 1,
\end{eqnarray*}
and its optimal policy.

For every $f \in F$, by Kolmogorov forward equation (\ref{2.8}),
\begin{eqnarray}\label{4.2}
	\sum_{j=0}^{\infty}p'_f(0,i,t,j)v^j
	=\sum_{k=1}^{\infty}p_f(0,i,t,k)kv^{k-1} B(f(k);v),\ \ v\in [0,1].
\end{eqnarray}
For any $a\in \bar{A}$, the generating function is
\begin{eqnarray*}
	B(a;v)=\sum_{k=0}^{\infty}b_k(a)v^k,\ \ v\in [0,1]
\end{eqnarray*}
and by Theorem V.10.1 in \cite{H63}, $B(a;v)=0$ has the smallest nonnegative root, $\rho(a)$.  Because $\bar{A}$ is finite, there exists $a_* \in \bar{A}$ such that
$\rho(a_*)=\min\limits_{a \in \bar{A}}\rho(a)=:\rho_* \in [0,1]$. Thus for any $a \in \bar{A}$,
\begin{eqnarray}\label{4.2.2}
	B(a; v) \geq B(a_*; \rho_*)=0,\ \ v\in [0,\rho_*].
\end{eqnarray} 
	
	On the other hand, for every $f \in F$, the extinction probability  $(ep_i^f : i \geq 1)$ satisfies $
	ep^f_i=\alpha^f_{i,m} ep^f_m\ (i \geq m+1)$,
	where $\alpha^f_{i,m}:=P^f_i(\tau_m<\infty)$ and $\tau_m:=\inf \{t\geq 0: X(t)=m\}$.
	
\begin{lem}\label{le4.1}
\rm{For every $f\in F$, $\alpha^f_{i,m}\geq \rho_*^{i-m},\ i\geq m+1$. Furthermore,  $\alpha^{f}_{i,m}=\rho_*^{i-m}$  if and only if $f(i)\in \{a\in \bar{A}:\rho(a)=\rho_*\}$ for all $i\geq m+1$.}
\end{lem}
\begin{proof}	
Let $\tilde{p}_f(s,i,t,k)$ be the transition probability of $\{X_{t \wedge \tau_m}:\ t \geq 0\}$.
It follows from (\ref{4.2}) that for any $i\geq m+1$,
\begin{eqnarray}\label{L-P}
	\sum\limits_{j=m}^{\infty}\frac{\partial \tilde{p}_f(0,i,t,j)}{\partial t}v^j=\sum\limits_{k=m+1}^{\infty}k v^{k-1}\tilde{p}_f(0,i,t,k)
	B(f(k);v),\ \ v\in [0,1),
\end{eqnarray}
where $B(f(k);v)=\sum\limits_{j=0}^{\infty}b_{j}(f(k))v^j$. Since $B(f(k);v)\geq 0$ for $v\in [0,\rho_*)$, integrating the above equality, using dominated convergence theorem and Fubini theorem; see \cite{AN72,C21,EV-14,ES-17,WC-18}. We will study it in this section., then letting $t \rightarrow \infty$, we yield that for any $i\geq m+1$,
\begin{eqnarray*}
&&\lim_{t \rightarrow \infty}\sum_{j=m}^{\infty}\tilde{p}_f(0,i,t,j)v^j-v^i\\
&=& \sum_{k=m+1}^{\infty}\int_0^{\infty}\tilde{p}_f(0,i,t,k)dtB(f(k);v)kv^{k-1}\\
&\geq& \sum_{k=m+1}^{\infty}\int_0^{\infty}\tilde{p}_f(0,i,t,k)dtB(a_*; \rho_*)kv^{k-1} = 0, \ \ v\in [0,\rho_*).
\end{eqnarray*}
Moreover, by \cite[P. 107]{AN72}, each state $j \geq m+1$ is transient. Hence,
	\begin{eqnarray*}
		\lim_{t \rightarrow \infty}\tilde{p}_f(0,i,t,m)v^m-v^i \geq 0\quad \forall v \in [0,\rho_*],\ i\geq m+1.
	\end{eqnarray*}
Letting $v=\rho_*$ yields
	$\lim\limits_{t \rightarrow \infty}\tilde{p}_f(0,i,t,m)\rho_*^m-\rho_*^i \geq 0$ for $i\geq m+1$.
Therefore, $\alpha^f_{i,m}\geq\rho_*^{i-m},\ \ i\geq m+1$.
Finally, the last assertion is obvious.
\end{proof}
Let
$F_m:=\{f \in F:\ \rho(f(i))=\rho_*\ \text{for\ all}\ i \geq m+1\}$.
By Theorem \ref{th3.1} and Lemma \ref{le4.1}, the optimal policy $f^*$ is  in $F_m$. Thus, we only need to consider policies in $F_m$. However, $F_m$ can be an infinite set of policies since there may be multiple $a\in \bar{A}$ satisfying $\rho(a)=\rho_*$. For convenience, we choose a fixed $a_*\in \bar{A}$ such that $\rho(a_*)=\rho_*$ and denote
\begin{eqnarray*}
F_m(a_*)=\{f \in F:\ f(i)=a_*\ \text{for\ all}\ i \geq m+1\}.
\end{eqnarray*}
Obviously, $F_m(a_*)$ is a finite subset of $F_m$. By Theorem \ref{th3.1} and Lemma \ref{le4.1}, $F_m(a_*)$ and $F_m$ are same in sense of minimizing the extinction probability $(ep_i^f:i\geq 1)$. Therefore, in the following, we only consider policies in $F_m(a_*)$.

The following lemma will be used in the proof of Theorem \ref{th4.1}.

\begin{lem}\label{le4.2a}
	{\rm Let $U=(u_{ij})_{n\times n}$ be a nonnegative $n\times n$ matrix satisfying that
		(1)\ $u_{ii}=u_{ij}=0$ for all $i\geq 1, j <i-1$,
		(2)\ $\sum\limits_{j=2}^nu_{1j}<1$, $\sum\limits_{j=i-1}^nu_{ij}\leq 1$ for all $2\leq i\leq n$ and
		(3) $u_{i i-1}>0$ for all $2\leq i\leq n$.
		Then $I_n-U$ is invertible ($I_n$ is a $n \times n$ identity matrix).
}
\end{lem}

\begin{proof}
	By mathematical induction,  we can complete the proof.
\end{proof}

The following theorem reveals that there are two cases: the extinction probability $(ep_i^f :\ 1 \leq i \leq m)$ is the unique solution of the following equation (\ref{4.4b}) or (\ref{4.4}).  Denote $L(i,a)=\sum\limits_{j=m}^{\infty}\tilde{p}(j|i, a)\rho_*^{j-m}$ for $1\leq i\leq m$ and $a\in A(i)$.

\begin{thm}\label{th4.1}
\rm{Suppose $f \in F_m(a_*)$.
\begin{description}
	\item[(i)]\ If $\min\limits_{1\leq i\leq m}b_0(f(i))>0$, then
			$(ep_i^f :\ 1 \leq i \leq m)$ is the unique solution of
\begin{eqnarray}\label{4.4b}
\begin{cases}
x_i=\tilde{p}(0|i, f(i))+\sum\limits_{j=1}^{m-1}\tilde{p}(j|i, f(i))x_j+L(i,f(i))x_m,\ & 1 \leq i \leq m\\
					0\leq x_i\leq 1,\ & 1 \leq i \leq m,
\end{cases}
\end{eqnarray}
where $\rho_*=\min\limits_{a\in\bar{A}}\rho(a)$.
Furthermore,  $(ep_i^f :\ i \geq m+1)$ satisfies
\begin{eqnarray}\label{4.6b}
				ep_i^f = \rho_*^{i-m}ep_m^f,\quad i \geq m+1.
\end{eqnarray}
			
\item[(ii)] If $\min\limits_{1\leq i\leq m}b_0(f(i))=0$, then
			$(ep_i^f :\ 1 \leq i \leq i_0-1)$ is the unique solution of
\begin{eqnarray}\label{4.4}
\begin{cases}
x_i=\tilde{p}(0|i, f(i))+\sum\limits_{j=1}^{i_0-1}\tilde{p}(j|i, f(i))x_j,\ & 1 \leq i \leq i_0-1,\\
0\leq x_i\leq 1,\ & 1 \leq i \leq i_0-1,
\end{cases}
\end{eqnarray}
where $i_0:=\min\limits_{1\leq i\leq m}\{i:b_0(f(i))=0\}$.
Furthermore, $ep_i^f=0,\ \text{for\ all}\ i\geq i_0$.
\end{description}	}
\end{thm}
\begin{proof}
	See the Appendix.
\end{proof}

We now describe the optimal policies in $F_m(a_*)$. For this purpose, let $\mathscr{M}:=\{i:1\leq i\leq m, \min\limits_{a\in A(i)}b_0(a)=0\}$. Define
\begin{eqnarray}\label{m_*}
	m_*=\min\limits\{i:i\in \mathscr{M}\}\ \text{if}\ \mathscr{M}\neq \emptyset\ \text{or}\ m+1\ \text{if}\ \mathscr{M}=\emptyset.
\end{eqnarray}

The following Theorem~\ref{th4.2} presents the sufficient and necessary condition of an optimal policy in two different cases: $m_*=m+1$ and $m_*\geq m$.

\begin{thm}\label{th4.2}
	\rm{
\begin{description}
\item[(i)]\  If $m_*=m+1$, then $\tilde{f}\ (\in F_m(a_*))$ is optimal if and only if $(ep^{\tilde{f}}_i: 1 \leq i \leq m)$ is the unique solution of the optimality equation (OE-1)
\begin{eqnarray}\label{4.13}
\begin{cases}x_i=\min\limits_{a \in A(i)}\{\tilde{p}(0|i, a)+\sum\limits_{j=1}^{m-1}\tilde{p}(j|i, a)x_j+L(i,a)x_m\}, \ &1 \leq i \leq m,\\
					0 \leq x_i \leq 1, & 1 \leq i \leq m,
\end{cases}
\end{eqnarray}
and $ep_i^{\tilde{f}}=\rho_*^{i-m}ep_m^{\tilde{f}},\ i \geq m+1$.

	\item[(ii)]\  If $m_*\leq m$, then $\tilde{f}\ (\in F_m(a_*))$ is optimal if and only if $(ep^{\tilde{f}}_i: 1 \leq i \leq m_*-1)$ is the unique solution of the optimality equation (OE-2)
\begin{eqnarray}\label{4.13a}
\begin{cases}x_i=\min\limits_{a \in A(i)}\{\tilde{p}(0|i, a)+\sum\limits_{j=1}^{m_*-1}\tilde{p}(j|i, a)x_j\},\ & 1 \leq i \leq m_*-1,\\
					0 \leq x_i \leq 1, \ & 1 \leq i \leq m_*-1,
\end{cases}
\end{eqnarray}
and $ep^{\tilde{f}}_i=0$ for all $ i \geq m_*$.
\end{description}	}
\end{thm}
\begin{proof}
	See the Appendix.
\end{proof}

\subsection{Improved policy iteration algorithm}
So far, the uniqueness of the solution of (\ref{4.13}) and (\ref{4.13a}) have been given, under which, this subsection considers an improved-policy iteration algorithm. To strictly improve the policies in $F_m$, we present the following theorem.

\begin{thm}\label{th4.3}
	\rm{Given $f \in F_m(a_*)$. For $1 \leq i \leq m$,
\begin{description}
	\item[(i)]\  if $m_*=m+1$, then let
			\small{\begin{eqnarray*}
				A^1_f(i)\!:=\!
				\{a \in A(i):\! ep_i^f>\tilde{p}(0|i, a)\!+\!\sum_{j=1}^{m-1}\tilde{p}(j|i, a)ep^f_j \! +\! L(i,a)ep^f_m\}.
\end{eqnarray*}}
Define a policy $\tilde{f}$ as follows: for $1 \leq i \leq m$, 		$\tilde{f}(i):=f(i)$ if $A^1_f(i)=\emptyset$ and $\tilde{f}(i)\in\mathop{\rm{argmin}}_{a\in A^1_f(i)}\{\tilde{p}(0|i, a)+\sum_{j=1}^{m-1}\tilde{p}(j|i, a)ep^f_j+L(i,a)ep^f_m\}$ if $A^1_f(i)\neq \emptyset$.
			Then $ep_i^{\tilde{f}}\leq ep_i^f (1 \leq i \leq m)$. Moreover, if $\tilde{f} \neq f$, then $ep_i^{\tilde{f}}<ep_i^f$ for some $1 \leq i \leq m$.
			
\item[(ii)] If $m_* \leq m$, then let
			\begin{eqnarray*}
				A^2_f(i):=\{a \in A(i):\ ep_i^f>\tilde{p}(0|i, a)+\sum_{j=1}^{m_*-1}\tilde{p}(j|i, a)ep^f_j\}.
			\end{eqnarray*}
			Define a policy $\tilde{f}$  as follows: for $1 \leq i \leq m_*$, $\tilde{f}(i):=f(i)$ if $A^2_f(i)=\emptyset$ and $\tilde{f}(i)\in\mathop{\rm{argmin}}_{a\in A^2_f(i)}\{\tilde{p}(0|i, a)+\sum_{j=1}^{m_*-1}\tilde{p}(j|i, a)ep^f_j\}$ if $A^2_f(i)\neq \emptyset$.
			Then $ep_i^{\tilde{f}}\leq ep_i^f (1 \leq i \leq m_*)$. Moreover, if $\tilde{f} \neq f$, then $ep_i^{\tilde{f}}<ep_i^f$ for some $1 \leq i \leq m_*$.
\end{description}}
\end{thm}
\begin{proof} First prove (i). For $m_*=m+1$, if $A^1_f(i) \neq \emptyset$, then $ep_i^f>\tilde{p}(0|i, \tilde{f}(i))+\sum\limits_{j=1}^{m-1}\tilde{p}(j|i, \tilde{f}(i))ep^f_j+L(i,\tilde{f}(i))ep^f_m$. If $A^1_f(i)=\emptyset$, then $f(i)=\tilde{f}(i)$ and hence
	$ep_i^f=\tilde{p}(0|i, \tilde{f}(i))+\sum\limits_{j=1}^{m-1}\tilde{p}(j|i, \tilde{f}(i))ep^f_j+L(i,\tilde{f}(i))ep^f_m$.
	Thus, for all $1 \leq i \leq m$,
	$$
	ep_i^f\geq\tilde{p}(0|i, \tilde{f}(i))+\sum\limits_{j=1}^{m-1}\tilde{p}(j|i, \tilde{f}(i))ep^f_j+[\sum\limits_{j=m}^{\infty}\tilde{p}(j|i, \tilde{f}(i))\rho_*^{j-m}]ep^f_m.
	$$
	Therefore, by Lemma \ref{le3.1}, we have $ep_i^f \geq ep_i^{\tilde{f}}$ for all $1 \leq i \leq m$.
	
	If $\tilde{f}\neq f$, then $A^1_f(i) \neq \emptyset$ and $\tilde{f}(i)\in A^1_f(i)$ for some $1 \leq i \leq m$. Hence, $ep_i^f>ep_i^{\tilde{f}}$ for some $ 1 \leq i \leq m$.
	
(ii) can be proved by an argument similar as (i).
\end{proof}

We now present a novel algorithm of an optimal policy in $F_m(a_*)$ and the minimal extinction probability
$$
EP^*:=(ep_1^*,\ldots, ep_m^*, ep_{m+1}^*,\ldots).
$$
For this end, we make some notations. First, denote $EP^f_{m_*-1}:=(ep_i^{f}:\ 1 \leq i \leq m_*-1)$, $\tilde{P}_{m_*-1}^f=(\tilde{p}(0|i,f(i)): 1 \leq i \leq m_*-1)$ and $\tilde{P}_{1,m_*-1}^f:=(\tilde{p}(1|i, f(i)), \cdots, \tilde{p}(m-1|i, f(i)),\sum\limits_{j=m}^{\infty}\tilde{p}(j|i, f(i))\rho_*^{j-m}:1\leq i\leq m_*-1)$.

(i) If $m_*=m+1$, then for $f \in F_m(a_*)$. Therefore, the unique solution of
\begin{eqnarray*}
	\begin{cases}
		x_i=\tilde{p}(0|i, f(i))+\sum\limits_{j=1}^{m-1}\tilde{p}(j|i, f(i))x_j+L(i,f(i))x_m,\ & 1 \leq i \leq m,\\
		0\leq x_i\leq 1,\ & 1 \leq i \leq m,
	\end{cases}
\end{eqnarray*}
is given by
$ EP^f_{m}=(I_{m}-\tilde{P}_{1,m}^f)^{-1}\tilde{P}_{m}^f$.

(ii) If $m_*\leq m$, then for $f \in F_m(a_*)$ with $b_0(f(m_*))=0$, denote $EP^f_{m_*-1}:=(ep_i^{f}:\ 1 \leq i \leq m_*-1)$, $\tilde{P}_{m_*-1}^f=(\tilde{p}(0|i,f(i)): 1 \leq i \leq m_*-1)$ and $\tilde{P}_{1,m_*-1}^f:=(\tilde{p}(1|i, f(i)), \cdots, \tilde{p}(m_*-1|i, f(i)):1\leq i\leq m_*-1)$. Therefore, the unique solution of
\begin{eqnarray*}
	\begin{cases}
		x_i=\tilde{p}(0|i, f(i))+\sum\limits_{j=1}^{m_*-1}\tilde{p}(j|i, f(i))x_j,\ & 1 \leq i \leq m_*-1,\\
		0\leq x_i\leq 1,\ & 1 \leq i \leq m_*-1,
	\end{cases}
\end{eqnarray*}
is given by
$EP^f_{m_*-1}=(I_{m_*-1}-\tilde{P}_{1,m_*-1}^f)^{-1}\tilde{P}_{m_*-1}^f$.

\begin{alg}\label{alg4.1}
	\rm{An improved-policy iteration algorithm of an optimal policy:
\begin{description}
	\item[(0)]\  Get $m_*$ by (\ref{m_*}).
			
\item[(1)]\  Obtain $\rho(a)$ by computing the minimal nonnegative solution of $B(a;v)=0$ for every $a \in \bar{A}$. Get $\rho_*=\min\limits_{a \in \bar{A}}\rho(a)$, $a_*\in \bar{A}$ satisfying $\rho(a_*)=\rho_*$ and $F_m(a_*)=\{f\in F: f(i)=a_* \ \text{for\ all}\ i\geq m+1\}$.
			
\item[(2)]\  Set $n=0$ and choose $f_0$ from $F_m(a_*)$. If $m_* \leq m$, then go to step (2a); if $m_* =m+1$, then go to step 3.
\begin{description}
\item[(2a)]\ \ (Policy evaluation) Get the extinction probability $EP^{f_n}_{m_*-1}:=(ep_i^{f_n}:\ 1 \leq i \leq m_*-1)$ by
\begin{eqnarray*}
EP^{f_n}_{m_*-1}=(I_{m_*-1}-\tilde{P}_{1,m_*-1}^{f_n})^{-1}
					\tilde{P}_{m_*-1}^{f_n}.
\end{eqnarray*}
\ \ \ \item[(2b)]\ (Policy improvement) Choose $f_{n+1}\in F_m(a_*)$ such that
\small{\begin{eqnarray*}
\begin{cases}
f_{n+1}(i)\in \mathop{\rm{argmin}}\limits_{a\in A(i)}\{p(0|i,a)+\sum\limits_{j=1}^{m_*-1}p(j|i, a)
						ep_j^{f_n}\},&\text{if}\ 1 \leq i \leq m_*, A^1_{f_n}(i)\neq \emptyset,\\
						f_{n+1}(i)=f_n(i), & \text{otherwise}.
\end{cases}
\end{eqnarray*}}				
\ \ \ \item[(2c)]\ If $f_{n+1}=f_n$, then stop, and set the optimal policy $f^*=f_n$ and the minimal extinction probability
\begin{eqnarray*}
ep_i^*=\begin{cases}
		ep_i^{f_n},& \text{if}\  1 \leq i \leq m_*-1\\
						0,\quad & \text{if}\  i \geq m_*;
\end{cases}
\end{eqnarray*}
otherwise increment $n$ by $1$ and return to step (2a).
\end{description}

\item[(3)]\  (Policy evaluation) Get the extinction probability $EP^{f_n}_{m}:=(ep_i^{f_n}:\ 1 \leq i \leq m)$ by
			\begin{eqnarray*}
				EP^{f_n}_{m}=(I_m-\tilde{P}_{1,m}^{f_n})^{-1}\tilde{P}_m^{f_n}.
			\end{eqnarray*}

\item[(4)]\  (Policy improvement) Choose $f_{n+1}\in F_m(a_*)$ such that
\small{\begin{eqnarray*}
\begin{cases}
f_{n+1}(i)\in \mathop{\rm{argmin}}\limits_{a\in A(i)}\{p(0|i,a)
					+\sum\limits_{j=1}^{m-1}p(j|i, a)
					ep_j^{f_n}+L(i,a)ep_m^{f_n}\},\quad & \text{if}\ 1 \leq i \leq m, A^2_{f_n}(i)\neq \emptyset\\
f_{n+1}(i)=f_n(i), & \text{otherwise}.
\end{cases}
\end{eqnarray*}}
			
\item[(5)]\  If $f_{n+1}=f_n$, then stop, and set the optimal policy $f^*=f_n$ and the minimal extinction probability
\begin{eqnarray*}
ep_i^*=\begin{cases}
		ep_i^{f_n},& \text{if}\  1 \leq i \leq m\\
		\rho_*^{i-m}ep_m^{f_n},\quad & \text{if}\  i \geq m+1;
\end{cases}
\end{eqnarray*}
otherwise increment $n$ by $1$ and return to step 3.
\end{description}	}
\end{alg}

\begin{rem}\label{rem4.1} \rm{
 Since $A(i)=\bar{A}\ (i\geq m)$ is a finite, $F_m(a_*)$ can obtained in finitely many iterations. Furthermore, Algorithm \ref{alg4.1} terminates after
finitely many iterations. Indeed,
by the definition of $F_m(a_*)$ and Algorithm \ref{alg4.1}, we have $f_n(i)=a_*\ (i\geq m+1)$ for all $n\geq 1$ and $ep_i^{f_{n+1}} \leq ep_i^{f_n}$ for all $i \geq 1$ in both cases. If $f_{n+1}$ is strictly better than $f_n$ for every $n$, i.e., $ep^{f_{n+1}}_i<ep_i^{f_n}$ for some $1 \leq i \leq m$. It contradicts with the finiteness of $F_m(a_*)$.
}
\end{rem}
\section{ Appendix}
In this section, we presents the proofs of Theorems \ref{th4.1}-\ref{th4.2} in Section \ref{section4}.

A. Proof of Theorem \ref{th4.1}:
\begin{proof}
	Take $f \in F_m(a_*)$. By Lemma \ref{le3.1} and Lemma \ref{le4.1}, it can be seen that $(ep_i^f: 1 \leq i \leq m)$ is a nonnegative solution of
\small{	\begin{eqnarray}\label{4.9}
		\begin{cases}
			x_i \!=\! \tilde{p}(0|i, f(i))\!+\!\sum\limits_{j=1}^{m-1}\tilde{p}(j|i, f(i))x_j \! +\! L(i,f(i))x_m,\ & 1 \leq i \leq m,\\
			0 \leq x_i \leq 1,\ & 1 \leq i \leq m.
		\end{cases}
	\end{eqnarray}}
	
	(i)\ If $\min\limits_{1 \leq i \leq m}b_0(f(i))>0$, then $b_0(f(i))>0$ for all $1 \leq i \leq m$ and thus $\tilde{p}(0|1, f(1))>0$. Then, we yield that
	\begin{eqnarray*}
		\begin{cases}
			\sum\limits_{j=1}^{m-1}\tilde{p}(j|1, f(1))+L(1,f(1))<1,\\
			\sum\limits_{j=1}^{m-1}\tilde{p}(j|i, f(i))+L(i,f(i))\leq1,\ \ 2 \leq i \leq m.
		\end{cases}
	\end{eqnarray*}
	Moreover, since $b_0(f(i))>0$ for all $1 \leq i \leq m$, we have $\tilde{p}(i-1|i, f(i))>0$. Let $\tilde{P}^{f}_{1,m}$ denote the coefficient matrix of the above system of linear equations. If $\rho_*>0$, then $\tilde{P}^{f}_{1,m}$ is irreducible and $(I_m-\tilde{P}^{f}_{1,m})^{-1}$ exists. If $\rho_*=0$, then $\tilde{P}^{f}_{1,m}$ satisfies the conditions in Lemma~\ref{le4.2a} and hence, $(I_m-\tilde{P}^{f}_{1,m})^{-1}$ also exists. Therefore, (\ref{4.13}) has unique nonnegative solution. Hence, by the definition of $F_m(a_*)$, (\ref{4.6b}) holds and thus (i) holds.
	
	(ii)\ If $\min\limits_{1 \leq i \leq m}b_0(f(i))=0$, then denote  $i_0:=\min\limits_{1 \leq i \leq m}\{i: b_0(f(i))=0\}$. By (\ref{q(j|i)}), we have $\tilde{p}(j|i, f(i))=\frac{q(j|i, f(i))}{q_i(f(i))}=\frac{b_{j-i+1}(f(i))}{b_1(f(i))}$ for every $j \geq i-1$ and $\tilde{p}(j|i, f(i))=0$ for $j \leq i-2$. Thus,
	for all path from every state $i \in [i_0, \infty)$ to state $0$, it must go through state $i_0$. By $b_0(f(i_0))=0$, we have $\tilde{p}(i_0-1 | i_0, f(i_0))=0$ and thus there is no path from state $i_0$ to state $i_0-1$. Therefore, $ep_i^f=0$ for every $i \geq i_0$.
	Moreover, because $(ep_i^f: 1 \leq i \leq m)$ is the nonnegative solution of (\ref{4.9}), $(ep_i^f:\ 1 \leq i \leq i_0-1)$ is the nonnegative solution of
	\begin{eqnarray}\label{K}
		\begin{cases}
			x_i=\tilde{p}(0|i, f(i))+\sum\limits_{j=1}^{i_0-1}\tilde{p}(j|i, f(i))x_j,\ & 1 \leq i \leq i_0-1,\\
			0\leq x_i\leq 1,\ & 1 \leq i \leq i_0-1.
		\end{cases}
	\end{eqnarray}
	It follows from $\sum\limits_{j=1}^{i_0-1}\tilde{p}(j|1, f(1))<1$ and Lemma~\ref{le4.2a} that (\ref{K}) has the unique nonnegative solution, that is, $(ep_i^f:\ 1 \leq k \leq i_0-1)$. Hence, (ii) holds.
\end{proof}
B. Proof of Theorem \ref{th4.2}:
\begin{proof}
	(i)\ Suppose that $m_*=m+1$. If $\tilde{f} \in F_m(a_*)$ is optimal, then $ep^{\tilde{f}}_i=ep^*_i$ for all $i \geq 1$, from which, together with Theorem \ref{th3.1} and Theorem \ref{th4.1}(i),  we have for $1 \leq i \leq m$,
	\begin{eqnarray*}
		\nonumber
		ep_i^{\tilde{f}}\!&=&\! ep^*_i \! =\! \min_{a\in A(i)}\{\tilde{p}(0|i, a)+\sum_{j=1}^{m-1}\tilde{p}(j|i, a)ep^*_j+L(i,a)ep^*_m\}\\
		\! &=&\! \min_{a\in A(i)}\{\tilde{p}(0|i, a)+\sum_{j=1}^{m-1}\tilde{p}(j|i, a)ep^{\tilde{f}}_j+L(i,a)ep^{\tilde{f}}_m\},
	\end{eqnarray*}
	i.e., $(ep^{\tilde{f}}_i: 1 \leq i \leq m)$ solves (\ref{4.13}). And by the definition of $F_m(a_*)$,
	we have $ep_i^{\tilde{f}} =ep_m^{\tilde{f}}\rho_*^{i-m}$.
	
	On the other hand, for given ${\tilde{f}} \in F_m(a_*)$, if $(ep^{\tilde{f}}_i: 1 \leq i \leq m)$ solves (\ref{4.13}), and $ep^{\tilde{f}}_i=\rho_*^{i-m}ep^{\tilde{f}}_m\ (i \geq m+1)$. Denote $f^* \in F_m(a_*)$ as an optimal policy.
	
	Since $\min\limits_{1\leq i\leq m}b_0(f(i))>0$, by Theorem \ref{th4.1}(ii) we get
	\begin{eqnarray*}
		ep^{\tilde{f}}_i \! =\! \tilde{p}(0|i, \tilde{f}(i))\!+\!\sum_{j=1}^{m-1}\tilde{p}(j|i, \tilde{f}(i))ep^{\tilde{f}}_j \! +\! L(i,\tilde{f}(i))ep^{\tilde{f}}_m
\end{eqnarray*}	
for $1 \leq i \leq m$.	Moreover, by Theorem \ref{th3.1}, for an optimal policy $f^*$ and every $1 \leq i \leq m$,
	\begin{eqnarray*}
		ep_i^{f^*}
		&=&\min\limits_{a \in A(i)}\{\tilde{p}(0|i, a)+\sum_{j=1}^{m-1}\tilde{p}(j|i, a)ep^{f^*}_j+L(i,a)ep^{f^*}_m\}\\
		&\leq&\tilde{p}(0|i, \tilde{f}(i))+\sum_{j=1}^{m-1}\tilde{p}(j|i, \tilde{f}(i))ep^{f^*}_j+L(i,\tilde{f}(i))ep^{f^*}_m.
	\end{eqnarray*}
Therefore, for $1\leq i\leq m$,
\begin{eqnarray}\label{ep-2}
		0 \leq ep_i^{\tilde{f}}-ep_i^{f^*}
		&\!\leq\!& \sum_{j=1}^{m-1}\tilde{p}(j|i, \tilde{f}(i))(ep_j^{\tilde{f}}-ep_j^{f^*})
+L(i,\tilde{f}(i))(ep_m^{\tilde{f}}-ep^{f^*}_m).
\end{eqnarray}	
Let $\tilde{P}_{1,m}^{\tilde{f}}$ be the coefficient matrix of the above system of linear inequalities. Since $\min\limits_{1\leq i\leq m}b_0(\tilde{f}(i))>0$, we have $\tilde{p}(i-1|i, \tilde{f}(i))>0$ for all $1 \leq i \leq m$, and hence $\sum_{j=1}^{m-1}\tilde{p}(j|1,\tilde{f}(1))+L(1,\tilde{f}(1))<1$.
Similar as in the proof of Theorem~\ref{th4.1},
we have that $(I-\tilde{P}_{1,m}^{\tilde{f}})^{-1}$ exists. Therefore, it follows from (\ref{ep-2}) that $ ep^{\tilde{f}}_i-ep^{f^*}_i=0$ for all $1\leq i\leq m$.
Hence, $\tilde{f}$ is optimal.
	
(ii)\ Suppose that $m_* \leq m$. Then there exists $\tilde{a}\in A(m_*)$ such that $b_0(\tilde{a})=0$. Hence, for every policy $f\in F_m(a_*)$ satisfying $f(m_*)=\tilde{a}$, since every path from $i\geq m_*$ to $0$ must go through state $m_*$ and there is no path from $m_*$ to $0$ under $f$, we get $ep^{f}_i=0\ (i\geq m_*)$.

If $\tilde{f} \in F_m(a_*)$ is optimal, then $ep^{\tilde{f}}_i=0\ (i\geq m_*)$, which deduces that $b_0(\tilde{f}(m_*))=0$ by the definition of $m_*$. By Theorem \ref{th4.1} (i),
\begin{eqnarray*}
ep^{\tilde{f}}_i=ep_i^*&=& \min\limits_{a \in A(i)}\{\tilde{p}(0|i,a)+\sum\limits_{j=1}^{m_*-1}\tilde{p}(j|i, a)ep_j^*\}=\min\limits_{a \in A(i)}\{\tilde{p}(0|i,a)+\sum\limits_{j=1}^{m_*-1}\tilde{p}(j|i, a)ep_j^{\tilde{f}}\}.
\end{eqnarray*}
Hence, $(ep^{\tilde{f}}_i:\ 1 \leq i \leq m_*-1)$ solves the OE-2, that is, (\ref{4.13a}), and
$ep^{\tilde{f}}_i=0$ for all $ i \geq m_*$.

On the other hand, suppose that $(ep^{\tilde{f}}_i:\ 1 \leq i \leq m_*-1)$ solves (\ref{4.13a}) and
$ep^{\tilde{f}}_i=0$ for all $ i \geq m_*$. Similar as the proof of (i), we can prove that  $ep^{\tilde{f}}_i=ep^{*}_i\ (1\leq i\leq m_*-1)$. Hence, $\tilde{f}$ is optimal.
\end{proof}

\section*{Acknowledgement}
This paper is supported by the National Natural Science Foundation of China (72342006) and the National Key Research and Development Program of China (2022YFA1004600).

\end{document}